\documentclass[12pt]{amsart}

\newtheorem{theorem}{Theorem}[section]

\newtheorem{corollary}[theorem]{Corollary}

\newtheorem{lemma}[theorem]{Lemma}
\newtheorem{proposition}[theorem]{Proposition}

\theoremstyle{definition}

\newtheorem{remark}[theorem]{Remark}

\newtheorem{example}[theorem]{Example}

\theoremstyle{parrafo}

\begin{document}

\title[]{Behavior of weak type bounds for high dimensional maximal operators defined by certain radial  measures}

\author{J. M. Aldaz and J. P\'erez L\'azaro}
\address{Departamento de Matem\'aticas,
Universidad  Aut\'onoma de Madrid, Cantoblanco 28049, Madrid, Spain.}
\email{jesus.munarriz@uam.es}
\address{Departamento de Matem\'aticas y Computaci\'on,
Universidad  de La Rioja, 26004 Logro\~no, La Rioja, Spain.}
\email{javier.perezl@unirioja.es}

\thanks{2000 {\em Mathematical Subject Classification.} 42B25}

\thanks{The authors were partially supported by Grant MTM2009-12740-C03-03 of the
D.G.I. of Spain}







\begin{abstract} As shown in \cite{A1}, the lowest constants appearing
in the weak type $(1,1)$ inequalities
satisfied by the centered Hardy-Littlewood maximal operator
associated to  certain finite radial measures,
grow exponentially fast with the dimension.
Here
we extend this result to a wider class of radial measures and to
 some values of $p > 1$. Furthermore, we improve the previously
known bounds
for $p =1$.
 Roughly speaking, whenever $p\in (1, 1.03]$,
if $\mu$ is
defined by a  radial, radially decreasing density satisfying some mild
growth conditions, then the best constants $c_{p,d,\mu}$ in the
weak type $(p,p)$ inequalities satisfy $c_{p,d,\mu} \ge 1.005^d$
for all $d$ sufficiently large. We also show that exponential
increase of the best constants occurs for certain families of
doubling measures, and for arbitrarily high values of $p$.
\end{abstract}


\maketitle


\section {Introduction}

\markboth{J. M. Aldaz and J. P\'erez-L\'azaro}{Behavior of weak type bounds}

Given a Borel measure $\mu$ on $\mathbb R^d$ and a locally integrable
function $g$, the Hardy-Littlewood maximal operator $M_{\mu}$ is given by
\begin{equation}\label{HLMF}
M_{\mu} g(x) := \sup _{\{r > 0: \mu (B(x, r)) > 0\}} \frac{1}{\mu
(B(x, r))} \int _{B(x, r)} \vert g\vert d\mu,
\end{equation}
where $B(x, r)$ denotes the euclidean {\it closed} ball of radius
$r > 0$ centered at $x$.  As is well known, $M_\mu$ is a positive, sublinear operator, acting  on
the cone of positive, locally integrable functions
($M_\mu$ is defined by using $|g|$ rather than $g$). The Hardy-Littlewood maximal operator admits many variants: Instead of
 averaging $|g|$ over balls centered at
$x$ (the centered operator) as in (\ref{HLMF}), it is possible to consider all balls containing $x$ (the uncentered
operator)  or average over  convex bodies more general than
euclidean balls (and even over more
general sets). And as part of the current effort to develop a
calculus on metric spaces, the Hardy-Littlewood maximal operator
has been studied in settings far more general than $\mathbb{R}^d$.
Here we work  with the {\em centered} operator defined
using {\em euclidean} balls in $\mathbb{R}^d$, associated to certain radial measures $\mu$ given by
$\mu (A) := \int_A  f(\|y\|_2)  d\lambda^{d} (y)$, where $f: (0,\infty ) \to [0,\infty )$ is nonincreasing
(possibly unbounded) and not zero almost everywhere, and $f(t) t^{d-1}\in L^1_{\operatorname{loc}}((0,\infty), dt)$. We emphasize
 that the  function
 $f$ defining $\mu$ is allowed to vary with the dimension $d$. Additional hypotheses,  regarding the growth at $0$ of $f$ and its decay
 at $\infty$,  are given  below.

The Hardy-Littlewood maximal operator is  an often used tool
 in  Real and Harmonic Analysis,
mainly (but not exclusively) due to the  fact that while
$|g|\le M_\mu g$ a.e.,
  $M_\mu g$  is
 not too large (in an $L^p$ sense)  since it satisfies the following strong type $(p,p)$ inequality: $\|M_\mu g\|_p \le C_p \|g\|_p$ for  $1 < p \le \infty$.
 For $p= 1$, $M_\mu$ satisfies instead the weak type $(1,1)$ inequality
 $\sup_{\alpha > 0}\alpha \mu (\{M_\mu g \ge \alpha\}) \le c_1 \|g\|_1$.
 Another aspect of the maximal operator that is receiving increasing
 attention, but not touched upon here, is that of its regularity properties, cf. for instance
 \cite{AlPe1},  \cite{AlPe2},  \cite{AlPe3} and the references contained
 therein.
When  $\mu = \lambda^d$, the  $d$-dimensional Lebesgue measure,
we often simplify notation, by writing $M$ rather than $M_{\lambda^d}$
and $dx$ instead of $d \lambda^d (x)$.

 Considerable efforts
have gone into determining how changing the dimension of $\mathbb{R}^d$ modifies the
best constants appearing in the weak and strong type inequalities.
When $p =\infty$, we can take $C_p = 1$ in every dimension $d$, since
averages never exceed a supremum. Quite remarkably,
E. M. Stein showed that for $M$,
there exist bounds for $C_p$ that are independent of $d$ (\cite{St1}, \cite{St2},
\cite{StSt}, see also \cite{St3}).
Stein's result was
generalized to the maximal function defined using an arbitrary norm by J. Bourgain (\cite{Bou1},
\cite{Bou2}, \cite{Bou3}) and A. Carbery (\cite{Ca}) when $p>3/2$.
For $\ell_q$ balls, $1\le q <\infty$, D. M\"{u}ller \cite{Mu}
showed that uniform bounds again hold for every $p
> 1$ (given $1\le q <\infty$, the $\ell_q$ balls
are defined using the norm $\|x\|_q :=\left( x_1^q+ x_2^q+\dots + x_d^q\right)^{1/q}$). It is still an open question whether the
maximal operator associated to cubes and Lebesgue measure is uniformly bounded
for $1 < p \le 3/2$.

When $p=1$, the maximal operator is (typically) unbounded, so
one considers weak type $(1,1)$ inequalities instead. In \cite{StSt} E. M. Stein
and J. O. Str\"{o}mberg proved that the smallest
constants in the weak type (1,1) inequality satisfied by $M$ grow
at most like $O(d)$ for euclidean balls, and at most like $O( d\log d)$
for more general balls. They also asked if uniform bounds could be found,
a question  still open for euclidean balls. But for cubes the  answer
is negative, cf. \cite{A2}.  In \cite{Au}, G. Aubrun refined the result from \cite{A2} by
 showing that $c_{1,d} \ge \Theta(\log^{1 -\varepsilon}d)$,
where $\Theta$ denotes the exact order and $\varepsilon > 0$ is arbitrary.
A very significant extension of the Stein and Str\"{o}mberg's
$O( d\log d)$ result, beyond the
euclidean setting, has recently been  obtained  by A. Naor and T. Tao, cf. \cite{NaTa}.

The weak type $(1,1)$ case  for integrable radial
densities defined via bounded decreasing functions was studied in \cite{A1}.
It was shown there that the best constants $c_{1,d}$
satisfy $c_{1,d} \ge \Theta\left(1\right)\left(
2/\sqrt 3\right)^{d/6}$, in strong
contrast with the linear $O(d)$ bounds known for $M$.
This suggests that for these measures and sufficiently
small values of $p > 1$, lack of uniform  bounds in $d$
should also hold. We show here that this is indeed the case, and for a wider class of measures
than those considered in \cite{A1}. We shall remove the assumption of boundedness on densities and
the assumption of finiteness on
measures, replacing these hypotheses with milder
growth conditions on the relative size of balls centered at the
origin (a possibility suggested in \cite[Remark 2.6]{A1}).  Instead of working directly with norm
(or strong type) inequalities when $p > 1$, we shall consider the
weak type $(p,p)$ inequalities. This allows us
to treat the cases  $p=1$  and $p > 1$ simultaneously. Needless
to say, lower bounds for weak type constants immediately imply the
same bounds for strong type constants. In Theorem
\ref{theoremdecp} we show that if balls centered at zero grow sufficiently
fast for some given radius, and this growth experiences a
certain rate of decay at infinity (cf. the theorem for the
exact technical conditions) then there is exponential increase
of the best constants $c_{p,d}$ in the weak type $(p,p)$ inequalities,
for every $p \in [1, p_0)$, where $p_0\approx 1.0378$.
The proof follows the lines of \cite{A1}, but replacing the Dirac delta $\delta_0$ with
$\chi_{B(0,v)}$ for some suitably chosen radius $v$, and using
 a better ball decomposition.
 This allows us to improve the bound on
$c_{1,d}$ from \cite{A1} to
$c_{1,d} \ge \Theta\left(1\right)\left(2/55^{1/6} \right)^{d}$
for $p=1$, even though we are considering
the characteristic function of a ball, rather than the more efficient $\delta_0$. Of course,  Dirac Deltas cannot be
used  when $p > 1$, since
the only reasonable definition of $\|\delta\|_p$
for $p > 1$ is $\|\delta\|_p = \infty$.

Exponential dependency
on the dimension of the best constants {\em also} holds for certain collections
of doubling measures and arbitrarily high values of $p$, cf.
Theorem \ref{theoremdoubling} below. Thus, Stein result on uniform
$L^p$ bounds for $M$ does not extend to
arbitrary doubling measures on $\mathbb{R}^d$, even though the class of doubling measures
represents a natural generalization of $\lambda^d$.
To highlight the difference  between $\lambda^d$ and the measures considered here, we point out  that when $M$ acts on
 radial, radially decreasing $L^p$ functions,
the best weak type $(p,p)$ constants $c_{p,d}$ equal $1$ in every
dimension, see \cite[Theorem 2.6]{AlPe4} (actually, the
result is stated there for $c_{1,d}$, but $c_{1,d} \ge c_{p,d}$,
cf. (\ref{weaktypequp}) below and the explanations afterwards),
while the best strong type constants satisfy $C_{p,d} \le 2^{1/q} q^{1/p}$,
where $q = p/(p-1)$, see \cite[Corollary 2.7]{AlPe4}.

Professor Fernando Soria informs us that he and Alberto Criado have
also
extended the results from \cite{A1} to
some values of $p > 1$, cf. \cite{Cr}; we mention that
where \cite{Cr} and this paper overlap,
the results presented here are more general
and give better bounds.

\section
{Notation and background results.}
The   restriction of  $\mu$ to a measurable set $A$ is denoted
by $\mu|_{A}$; that is,  $\mu|_{A} (B) = \mu (A\cap B)$.
We always assume that $\mu (\mathbb{R}^d) > 0$ and $\mu (B(x, r)) <\infty$,
i.e., measures are nontrivial and locally finite.
The maximal function of a locally finite measure $\nu$ is
defined by
\begin{equation}\label{HLMFmeas}
M_{\mu} \nu (x) := \sup _{\{r > 0: \mu (B(x, r)) > 0\}} \frac{\nu (B(x, r))}{\mu (B(x, r))}.
\end{equation}
 Note  that
formula (\ref{HLMF}) is simply (\ref{HLMFmeas}) in the special case   $\nu << \mu$. Our choice of closed balls in (\ref{HLMF}) and
(\ref{HLMFmeas})
is mere convenience; using open balls instead does not
change the value of the maximal operator at $x$, since each
closed ball is a countable intersection of open balls.
The boundary of $B(x, r)$
is the sphere $\mathbb S (x,r)$. Sometimes we use $B^d (x, r)$ and
$\mathbb S^{d-1} (x,r)$ to make their dimensions explicit. If
$x=0$ and $r=1$, we use the abbreviations $B^d$ and $\mathbb S^{d-1}$.
 Balls are defined using
the $\ell_2$ or euclidean distance
$\|x\|_2 := \sqrt{x^2_1 + \dots +x_d^2}$. The Lebesgue measure on $\mathbb R^d$ is denoted by $\lambda^d$, and area measure on a $d-1$ sphere,
by $\sigma^{d-1}$. Sometimes it is convenient to use normalized versions
of these measures, so balls and spheres  have total mass 1; we
use $N$ as a subscript to denote these normalizations. Thus,
$\lambda^d_N (B^d) = 1$ and $\sigma^{d-1}_N(\mathbb S^{d-1})= 1$.

Regarding the relationships
between different constants, let us
recall that by
the Besicovitch Covering Theorem,  for every locally finite Borel measure
$\mu$ on $\mathbb{R}^d$, and  every $p$ with $1 \le p < \infty$,
the maximal operator satisfies the following
weak type $(p,p)$ inequality:
\begin{equation}\label{weaktypep}
\mu (\{M_\mu g \ge \alpha\}) \le \left(\frac{c \|g\|_p}{\alpha}\right)^p,
\end{equation}
where $c=c(p,d, \mu)$ depends neither on $g\in L^p (\mathbb R^d, \mu)$
nor on $\alpha > 0$. The constant $c$ can also
be taken to be independent of $\mu$ and of $p$. Set $q := p/(p-1)$. Using the quantitative version of the
Besicovitch Covering Theorem given in
 \cite[p. 227]{Su}, we have
 \begin{equation}\label{weaktypequ}
\mu (\{M_\mu g \ge \alpha\}) \le \frac{(2.641 + o(1))^d\|g\|_1}{\alpha}.
\end{equation}
Thus, if $g\in L^p (\mathbb R^d, \mu)$, then $|g|^p \in L^1 (\mathbb R^d, \mu)$, and it follows from Jensen's inequality that
  \begin{equation}\label{weaktypequp}
\mu (\{M_\mu g \ge \alpha\}) = \mu (\{(M_\mu g)^p \ge \alpha^p\}) \le \mu (\{M_\mu |g|^p \ge \alpha^p\}) \le \frac{(2.641 + o(1))^d\|g\|_p^p}{\alpha^p}.
\end{equation}
Letting $c_{p,d,\mu}$ be the best constant $c$ in (\ref{weaktypep}), we
have
$c_{p,d,\mu} \le (2.641 + o(1))^{d/p}$.  This bound is uniform in $\mu$, and
setting $p=1$ in the exponent $d/p$, it can be made uniform in $p$
also. Replacing $(2.641 + o(1))^{d/p}$ by $c_{1,d, \mu}$ in the right hand side
of (\ref{weaktypequp}), we also obtain $c_{p,d,\mu} \le (c_{1,d,\mu})^{1/p} \le c_{1,d, \mu}$.
Let $C_{p,d,\mu}$ be the lowest constant
in
\begin{equation}\label{strongtypep}
\|M_\mu g\|_{L^p(\mathbb{R}^d, \mu)} \le  C_{p,d,\mu} \|g\|_{L^p(\mathbb{R}^d, \mu)}.
\end{equation}
It is
an immediate consequence of Chebyshev's inequality that
$c_{p,d,\mu} \le  C_{p,d,\mu}$, since
\begin{equation}\label{weaklestrong}
\alpha^p\mu (\{(M_\mu g)^p \ge \alpha^p\}) \le \|M_\mu g\|_{p}^p \le C_{p,d, \mu}^p \|g\|_{p}^p.
\end{equation}
When $p$ is small, lower bounds for
$c_{p,d,\mu}$ are quite often not just formally stronger, but substantially
stronger than lower bounds for $C_{p,d,\mu}$,
since it is well-known that for many measures $C_{1,d,\mu} =\infty$
and
$\lim_{p\to 1} C_{p,d,\mu} = \infty$, while  $c_{p,d,\mu}\le c_{1,d,\mu} \le (2.641 + o(1))^{d}$.

Let $d >>1$, and consider Lebesgue measure restricted to the
 unit ball. Most of its mass is concentrated near $\mathbb{S}^{d-1} (0, 1)$, since volume scales like $R^d$, so the ball ``looks" very much like the sphere.  The main idea in \cite{A1} and here is to realize that
 this is a rather general phenomenon: Rotationally invariant
 measures with a certain decay at infinity, will often be very similar
in a
certain region to area
 on some sphere $\mathbb{S}^{d-1} (0, R_1)$. Hence, the size of  balls
 in that region can be estimated by intersecting them with $\mathbb{S}^{d-1} (0, R_1)$  and then using
 the area of the spherical caps resulting from such intersections.
 Given a unit vector $v\in \mathbb
R^{d}$ and $s \in [0, 1)$, the $s$ spherical
cap about $v$ is the set $C(s, v) :=\{\theta \in
\mathbb S^{d-1}: \langle \theta, v\rangle \ge s\}$.
Spherical caps are just geodesic balls $B_{\mathbb
S^{d-1}}(x, r)$ in $\mathbb S^{d-1}$. For spheres other than
$\mathbb{S}^{d-1}$, spherical caps are defined in an entirely analogous
way.
If $v=e_1 =
(1,0,\dots , 0)$ and $s =2^{-1}$, then
\begin{equation}\label{cap}
B_{\mathbb S^{d-1}}(e_1, \pi/ 3) =
C(2^{-1}, e_1) = \mathbb
S^{d-1} \cap B( e_1,  1).
\end{equation}
More generally, given any angle $r\in (0, \pi/2)$,
writing $s = \cos r$ and $t = \sin r$, we have
\begin{equation}\label{cap1}
B_{\mathbb S^{d-1}}(e_1, r) =
C(s, e_1) \subset  B( s e_1,  t).
\end{equation}
The following lemma shows that $\sigma^{d-1}_N(C(s, e_1)) = t^d/\Theta(\sqrt{d})$,
where $\Theta$ stands for exact order (i.e.,
$g = \Theta(h)$ if and only if $g = O(h)$ and $h= O(g)$); the
special case $r=\pi/3$ is used in
the proof of \cite[Theorem 2.3]{A1}.
We recall the following results on volumes
and areas: i)
$\lambda^d (B^d) =\frac{\pi^{d/2}}{\Gamma (1 + d/2)}$; ii)
$\sigma^{d-1} (\mathbb S^{d-1}) = d \lambda^d (B^d)$; iii)
 $\sigma^{d-1} (B_{\mathbb
S^{d-1}} (x,r)) = \sigma^{d-2} (\mathbb S^{d-2})\int_0^r
\sin^{d-2}t dt$ (cf. for instance \cite[(A.11) pg. 259]{Gra}).

\begin{lemma}\label{sphericalcaps} Let $r\in (0, \pi/2)$, let
$\sigma^{d-1}_N$
be normalized area on the sphere $\mathbb S^{d-1} (0, R)$, and let
$s = \cos r$, $t = \sin r$, so with this notation,
$\sigma^{d-1}_N (B_{\mathbb S^{d-1} (0, R) }(R e_1, R r))
 =
\sigma^{d-1}_N (C(R s, R e_1))$.
Then
\begin{equation}\label{caps}
\frac{ t^{d-1}}{ \sqrt{2 \pi d}}
 \le
\sigma^{d-1}_N (C(R s, R e_1))
\le  \frac{ t^{d-1}}{s \sqrt{2 \pi d}}\sqrt{1 + \frac1d}.
\end{equation}
\end{lemma}

\begin{proof} Observe first that the relative size of caps depends
 neither on the center of the ball nor on the radius.
In particular, since we are dealing with normalized area, we may
assume that $R=1$.
We  use the following Gamma function
estimate (an immediate consequence of the log-convexity of
$\Gamma$ on $(0,\infty)$, cf. Exercise 5, pg. 216 of \cite{Web}):
\begin{equation}\label{ratio}
\left(\frac{d}{2}\right)^{1/2}
\le
\frac{\Gamma (1 + d/2)}{\Gamma (1/2 + d/2)}
\le
\left(\frac{d+1}{2}\right)^{1/2}.
\end{equation}

From i), ii), iii), (\ref{ratio})
and the fact that $\cos u \ge s$ on $[0, r]$, we get:
\begin{equation}\label{in}
\sigma^{d-1}_N\left( C(s, e_1)\right)
\le
\frac{\sigma^{d-2}
(\mathbb S^{d-2})}{s \sigma^{d-1}(\mathbb S^{d-1})} \int_0^{r}
\sin^{d-2}u \cos u du
\end{equation}
\begin{equation}\label{in2}
= \frac{1}{sd} \frac{\lambda^{d-1} (B^{d-1})}{\lambda^{d} (B^{d})}
t^{d-1}
\le
\frac{ t^{d-1}}{s \sqrt{2 \pi d}}\sqrt{1 + \frac1d}.
\end{equation}
Likewise, since $\cos u \le 1$,
\begin{equation}\label{in}
\sigma^{d-1}_N\left( C(s, e_1)\right)
\ge
\frac{\sigma^{d-2}
(\mathbb S^{d-2})}{\sigma^{d-1}(\mathbb S^{d-1})} \int_0^{r}
\sin^{d-2}u \cos u du
= \frac{1}{d} \frac{\lambda^{d-1} (B^{d-1})}{\lambda^{d} (B^{d})}
t^{d-1}
\ge
\frac{ t^{d-1}}{ \sqrt{2 \pi d}}.
\end{equation}
\end{proof}

\section
{Weak type $(p,p)$ bounds for rotationally invariant measures}
 Fix $d\in \mathbb{N}\setminus \{0\}$, and
let $f: (0,\infty ) \to [0,\infty )$ be a nonincreasing
(possibly unbounded)
 function, not zero almost everywhere, such that $f(t) t^{d-1} \in
 L^1_{\operatorname{loc}}((0,\infty), dt)$. Then the function $f$
defines a  locally integrable, rotationally invariant (or radial) measure
$\mu$  on $\mathbb{R}^d$ via
\begin{equation}\label{defrad}
\mu (A) := \int_A  f(\|y\|_2)  d\lambda^{d} (y).
\end{equation}
Observe that the local integrability of $f(t) t^{d-1}$ is assumed for a {\em fixed} $d$, not
for all values of $d$ simultaneously. Note also that $f$ can
depend on $d$.
When $A = B(0, R)$, integration in polar coordinates
yields the well known expression $\mu (B(0,R)) = \sigma^{d-1} (\mathbb{S}^{d-1})\int_0^R f(t) t^{d-1} dt$.
Since  (unlike \cite{A1}) finiteness of measures and boundedness of
densities are not assumed in the present paper, we need to impose some conditions on the rate of growth of
balls centered at zero. To this end, we define,
for all $u\in (0,1]$ and all $R > 0$ such that $\mu (B(0, u R)) >0$,
\begin{equation}\label{funch}
h_u (R):= \frac{\mu (B(0, R))}{\mu (B(0, u R))}.
\end{equation}
In the extreme case $\mu = \delta_0$, $h_u (R) = 1$ always, and
for every $g$ with $g(0) < \infty$, we have
$M_\mu g = g = g(0)$ a.e. with respect to $\delta_0$.
Thus, for all $p\ge 1$ and all $d\ge 1$, $c_{p,d,\delta_0} =
C_{p,d,\delta_0} = 1$. Of course, in this case there is no relationship between
 $\delta_0$ and $d$.
 For the  measures  considered in \cite[Theorem 2.3]{A1},
$\lim_{R\to 0} h_u(R) =
  u^{- d}$ and $\lim_{R\to \infty} h_u(R) =
  1$; we present this fact, which appears within the proof of \cite[Theorem 2.3]{A1},
  as part of the next proposition.

\begin{proposition} \label{dimmeas} Fix $d \in \mathbb{N}\setminus\{ 0\}$.  Let $f: (0,\infty ) \to [0,\infty )$ be a nonincreasing
 function with $f > 0$ on some interval $(0,a)$
and $f(t) t^{d-1}\in L^1_{loc} (0, \infty)$. If
 $\mu$ is the measure defined by (\ref{defrad}), then for every
 $u\in (0,1)$ and every $R > 0$ we have $h_u (R) \le u^{-d}$.
 If additionally $f$ is bounded, then for every $u\in (0,1)$,
 $\sup_{R > 0} h_u (R) = \lim_{R\to 0} h_u (R) = u^{-d}$.
 Regardless of whether $f$ is bounded or not, if $\mu$ is finite, then
 for every
 $u\in (0,1)$ we have $\lim_{R\to \infty} h_u (R) = 1$.
\end{proposition}

\begin{proof} The fact that $\sup_{R > 0} h_u (R) \le u^{-d}$ is
obvious since $f$ is nonincreasing, so the case where $f$ is constant
yields the largest possible growth, and then we just have a
multiple of Lebesgue measure. Or, more formally:
\begin{equation}
\frac{\mu (B(0, R))}{\mu (B(0, u R))}
= \frac{\mu (B(0, u R)) + \mu (B(0, R)\setminus B(0, uR))}{\mu (B(0, u R))}
\end{equation}
\begin{equation}
=
1 + \frac{\sigma^{d-1} (\mathbb{S}^{d-1})\int_{uR}^R f(t) t^{d-1} dt}{\sigma^{d-1} (\mathbb{S}^{d-1})\int_0^{u R} f(t) t^{d-1} dt}
\le
1 + \frac{f(u R) \int_{uR}^R t^{d-1} dt}{f(u R) \int_0^{uR} t^{d-1} dt} = u^{-d}.
\end{equation}

Suppose next that in addition to being nonincreasing, $f$ is bounded.
Then the
averages $\frac{1}{\lambda^d (B(0,R))} \int_{B(0,R)} f(\|x\|_2) dx$ are bounded and nonincreasing with respect to $R$. Thus,
$\lim_{R\to 0} \frac{1}{\lambda^d (B(0,R))} \int_{B(0,R)} f(\|x\|_2) dx = L$ exists, and
\begin{equation}
\lim_{R\to 0} \frac{\mu (B(0, R))}{\mu (B(0, u R))}
=
\lim_{R\to 0} \frac{\int_{B(0,R)} f(\|x\|_2) dx}{\int_{B(0,u R)} f(\|x\|_2) dx}
= \lim_{R\to 0} \frac{ L \lambda^d (B(0,  R))}{L \lambda^d (B(0, u R))} = u^{-d}.
\end{equation}
The last assertion about finite measures is obvious.
\end{proof}

\begin{remark} The condition
$\lim_{R\to 0} h_u (R) = u^{-d}$ can be satisfied by unbounded densities
with a mild singularity a $0$. Consider, for instance, $f(x) = |\log (x) \chi_{(0,1]}(x)|$,
for every $d \ge 1$.
\end{remark}

\begin{lemma}\label{pbound} Let $\mu$ be a rotationally invariant measure
on $\mathbb{R}^d$,
 let $1\le p < \infty$, and let $q := p/(p -1)$.
 For $0 < R$ and  $0< v < 1$,
  write
 $H := \sqrt{R^2 + v^2 R^2}$.  If the  pair
 $(v, R)$ is such that $\mu (B(0, v R)) > 0$, then
\begin{equation}\label{lowbouwpgen}
c_{p, d, \mu}
\ge
\frac{\mu(B(0,v R))^{1/q} \mu (B(0,R))^{1/p}}{2 \mu (B(R e_1, H))}.
\end{equation}
If additionally there exist  $T, t_0 > 0$ and $v_0 \in (0,1)$
such that
 $
 \sup_{\{R >0: v_0 R \ge T\}} h_{v_0} (R) \le v_0^{ - t_0 d},
 $
 then
 \begin{equation}\label{lowbouwp}
c_{p, d, \mu} \ge
\sup_{\{R >0: v_0 R \ge T\}} \frac{ v_0^{t_0 d/q}\mu (B(0,R))}{2 \mu (B(R e_1, H))}.
\end{equation}
\end{lemma}

\begin{proof}
Note that
\begin{equation}\label{maxfnbouwp}
M_\mu \chi_{B(0,v R)} (R e_1)
\ge
\frac{\|\chi_{B(0, v R)}\|_1}{2 \mu (B(R e_1, H))} =: \alpha.
\end{equation}
By rotational invariance of $\mu$, we have $B(0, R) \subset \{M_\mu
\chi_{B(0,v R)} \ge \alpha\}$.
And since  $\chi_{B(0,v  R)} = \chi_{B(0,v R)}^p$, it follows
that
$\|\chi_{B(0,v R)}\|_1 = \|\chi_{B(0,v R)}\|_p^p$.
Using (\ref{weaktypep}) we see that
\begin{equation}\label{bouwp1}
c_{p, d, \mu}
\ge
\frac{ \alpha\left( \{M_\mu \chi_{B(0,v R)} \ge \alpha\} \right)^{1/p}}{\|\chi_{B(0,v R)}\|_p}
\ge
\frac{\|\chi_{B(0,v R)}\|_1}{2 \mu (B(R e_1, H))}
\frac{ \mu (B(0,R))^{1/p}}{\|\chi_{B(0, v R)}\|_p}
\end{equation}
\begin{equation}\label{bouwp2}
=
\frac{\mu(B(0,v R))^{1/q} \mu (B(0,R))^{1/p}}{2 \mu (B(R e_1, H))}.
\end{equation}
Specializing to $v=v_0$ and using the hypothesis on $T$ and $t_0$ we obtain
\begin{equation}\label{bouwp10}
c_{p, d, \mu}
\ge
  v_0^{t_0 d/q}
\frac{ \mu (B(0,R))}{2 \mu (B(R e_1, H))}
\end{equation}
for every $ R > 0$ such that $v_0 R \ge T$.
\end{proof}

The preceding Lemma is more general than needed in the present paper,
since we
are not assuming that $\mu$ is of the form given by (\ref{defrad});
this greater generality will be useful in future work. If $\mu$
is given by (\ref{defrad}), then by Proposition \ref{dimmeas},
the condition
on $h_v(R)$ is satisfied for some $t_0 \le 1$, all $v\in (0,1)$,
and all $T > 0$. So the Lemma is applicable and
furthermore, any $v_0\in (0,1)$ can be used
(in the next Theorem we take $v_0 = 1/2$). The idea of the
proof is to
choose $R_1$ so $ \mu (B(R_1 e_1, H))$ is exponentially small
(in $d$) when compared with  $\mu (B(0, R_1))$, and then to adjust
$q$ in (\ref{lowbouwp}) so $v_0^{-1/q}$ is sufficiently close to $1$. This
yields exponential
growth of the constants for  $p > 1$ small enough. Recall that $C_{p, d, \mu}$ denotes the best constant in the
strong type $(p,p)$ inequalities. We emphasize that in the next result,
we can have different functions $f$ associated to different dimensions
$d$.

\begin{theorem}\label{theoremdecp} Fix
$d\in \mathbb{N}\setminus \{0\}$,
and set $u = \sqrt{2/3}$.
 Let $f: (0,\infty ) \to [0,\infty )$ be a nonincreasing
 function and let $\mu$ be the radial measure defined via (\ref{defrad}). Assume  $\mu$ satisfies
 \begin{equation}\label{hipotesis}
   \sup_{R > 0} h_u (R) \ge
   u^{-\left(\frac{6\log 2 - \log55}{3\log 3 - 3\log2}\right) d}
  =
  \left(\frac{64}{55}\right)^{\frac{d}{6}}\ge \limsup_{R\to \infty} h_u(R).
 \end{equation}
  Then for every $p$ such that
$
 1 \le p < \frac{6\log 2}{\log 55}  \approx 1.0378 $,
we have
\begin{equation}\label{bestdec}
55^{-1/6}2^{1/p}>1
\mbox{\ \ \ \ \ and \ \ \ \ \ }
C_{p,d,\mu} \ge c_{p,d,\mu} \ge  \frac{1}{4 + \Theta\left(\frac{1}{\sqrt{d}}\right)} \left(\frac{2^{1/p}}{55^{1/6}}\right)^{d}.
\end{equation}
\end{theorem}

\begin{proof}  Assume that
$d \ge 2$, and set $H =  \sqrt{R_1^2 + v^2 R_1^2}$, as in Lemma \ref{pbound}. Arguing as in \cite[Theorem 2.3]{A1},
we look for a radius $R_1$ such that $B(R_1 e_1, H)$  has very small measure
compared to $B(0, R_1)$.
Fix $0<\varepsilon<1/10$. Define
$A= \{R > 0: h_u(R) \ge  (1-\varepsilon) (64/55)^{d/6}\}$. By the continuity in $R$ of $h_u$ and the hypotheses  in (\ref{hipotesis})  $A$ is a
nonempty closed set. If $A$ is unbounded, we choose $R_1\in A$ so large that $h_u(R_1) \ge (1-\varepsilon) (64/55)^{d/6}$ and $h_u(u^{-1}R_1),
h_u(u^{-2}R_1) < (1+\varepsilon)(64/55)^{d/6}$. If $A$ is bounded, then $R_1:=\max A$ automatically satisfies the preceding conditions on $h_u(R_1)$, $h_u(u^{-1}R_1)$, and $h_u(u^{-2}R_1)$.
Set $v = 1/2$. Then $H =  R_1 \sqrt{5}/2$.
Write $T = \sqrt{R_1^2 + H^2} = 3R_1/2$,
and observe that $T=u^{-2}R_1$,
so
$B(R_1 e_1, H)\cap \{x_1\le R_1\} \subset
 B(0,  3R_1/2)$. Since the density of $\mu$ is radially decreasing,
\begin{equation}\label{conc}
\mu (B(R_1 e_1, H)\cap \{x_1\ge R_1\})\le\mu (B(R_1 e_1, H)\cap \{x_1\le R_1\}).
\end{equation}
Thus $\mu (B(R_1 e_1, H) \le 2 \mu (B(R_1 e_1, H) \cap \{x_1\le R_1\})$, so is enough to control this latter term.
To this end, we split
$B(R_1 e_1, H) \cap \{x_1\le R_1\}$ into the following three  pieces and
estimate the measure of each one: $B(0, u R_1)\cap B(R_1 e_1, H)$,
$B(0, R_1)^c\cap B(R_1 e_1,  H) \cap \{x_1\le  R_1\}$, and
$(B (0, R_1)\setminus B(0, \sqrt{2/3} R_1)) \cap B(R_1 e_1, H)$.
 First we bound the
part containing the origin:
\begin{equation}\label{inside1dp}
\mu (B(0, u R_1)\cap B(R_1 e_1, H))
\le
\mu (B(0, u R_1))
\le
\frac{\mu (B(0, R_1))}{1-\varepsilon} \left(\frac{55}{64}\right)^{d/6}.
\end{equation}

The other two parts  are contained
inside
certain cones, whose radial projections into the unit sphere are spherical caps. So we apply Lemma \ref{sphericalcaps}.
To control
$\mu (B(0, R_1)^c\cap B(R_1 e_1,  H) \cap \{x_1\le  R_1\})$,
we define $\nu$ on
$\mathbb{S}^{d-1}(0, 1)$  as the pushforward
(via the radial projection map) of
 $\mu$ restricted to $B(0, u^{-2} R_1)\setminus B(0,  R_1)$.
 Now $\nu$ is a rotationally invariant measure on $\mathbb{S}^{d-1}$,
so it must be a  multiple  $m \sigma^{d-1}_N$ of normalized area.
Since $\nu (\mathbb{S}^{d-1}) = m$, we have
$m= \mu (B(0, u^{-2} R_1)\setminus B(0,  R_1))$
and thus
 $
 \nu = \mu\left(B(0, u^{-2} R_1)\setminus B(0,  R_1)\right)
\sigma^{d-1}_N.
$
We use symmetry to find the spherical cap $C$ determined by
the intersection of $\mathbb{S} (0, R_1)$ with $B(R_1 e_1, H)$,
restricting ourselves to the $x_1 x_2$-plane. Simultaneously
solving
$x_1^2 + x_2^2 = R_1^2$ and $(x_1 - R_1)^2 + x_2^2 = H^2$,
we find that the radial projection of $C$ into $\mathbb{S}^{d -1}$
is   $C(3/8, e_1)$. Now
\begin{equation}\label{halfout1p}
\mu (B(0, R_1)^c\cap B(R_1 e_1,  H))
\end{equation}
\begin{equation}\label{halfout2p}
= \mu (B(0, R_1)^c\cap B(R_1 e_1, H) \cap \{x_1 \le  R_1\})
+
\mu (B(R_1 e_1, H) \cap \{x_1 >  R_1\}).
\end{equation}
By  Lemma \ref{sphericalcaps} with $\cos r = 3/8$ (so $\sin r = \sqrt{55}/8$) and by the choice of $R_1$,
\begin{equation}\label{halfout3p}
\mu (B(0, R_1)^c\cap B(R_1 e_1, H) \cap \{x_1 \le  R_1\})
\le
 \mu (B(0, R_1)) (1+\varepsilon)^2 \left(\frac{64}{55}\right)^{d/3}
 \sigma^{d-1}_N (C(3/8, e_1))
\end{equation}
\begin{equation}\label{halfout4p}
\le
\mu (B(0, R_1))  \left(\frac{55}{64}\right)^{d/6}
\Theta\left(\frac{ 1}{\sqrt{d}}\right).
\end{equation}

Regarding the measure of
$(B (0, R_1)\setminus B(0, \sqrt{2/3} R_1)) \cap B(R_1 e_1, H)$,
this set
is contained in the (positive) cone subtended by the cap $C$
resulting from the intersection
of $\mathbb{S}^{d-1}(0, \sqrt{2/3} R_1)$ with
$B( R_1 e_1, H)$. The said cone is formed by all rays starting at $0$
and crossing $C$.
Let $r$ be the maximal angle
between a vector in
this cap and the $x_1$-axis. We consider
the intersection of $C$ with the
$x_1 x_2$-plane, in order to determine
$s := \cos r$ and
$t := \sin r$.
Solving $(x_1 - R_1)^2 + x_2^2 = H^2$ and
$ x_1^2 + x_2^2 = (\sqrt{2/3} R_1)^2$, we obtain
$t = \sqrt{1077}/(24\sqrt{2})$. Projecting
radially $\mu|_{B(0, R_1)}$  to
$\nu =  \mu (B(0, R_1)) \sigma^{d - 1}_N$
on
$\mathbb{S}^{d-1}$, and likewise projecting radially $C$
onto  $C (s , e_1)$, from Lemma \ref{sphericalcaps} we obtain
\begin{equation}\label{inside2decp}
\mu \left(\left(B (0, R_1)\setminus B(0, u R_1)\right) \cap B(R_1 e_1, H)\right) \le
\nu (C(s, e_1))
\le
 \mu (B(0, R_1)) t^{d}
O \left(\frac{1}{\sqrt{d}}\right).
\end{equation}
The preceding estimates, together with
$t = \sqrt{1077}/(24\sqrt{2})
<  (55/64)^{1/6}$,
entail that
\begin{equation}\label{halfin2ppar2}
\mu (B(R_1 e_1, H) \cap \{x_1 \le  R_1\})
\le
\mu (B(0, R_1))  \left(\frac{55}{64}\right)^{d/6}
\left(\frac{1}{1-\varepsilon} +  \Theta\left(\frac{ 1}{\sqrt{d}}\right)\right),
\end{equation}
and we already know from (\ref{conc})
that $\mu (B(R_1 e_1, H))$ is at most twice as large.
Since by Proposition \ref{dimmeas},
$\mu (B(0, R) )\le
v^{ - d} \mu (B(0, v R))$
 for every $v\in (0,1)$ and every
 $R > 0$, we can apply Lemma \ref{pbound} with
$t_0 = 1$, $R = R_1$ and $v_0 = 1/2$. This yields
\begin{equation}\label{constp}
c_{p,d, \mu} \ge \frac{ \left( 2^{1/p}55^{-1/6}\right)^d}{
\frac{4}{1-\varepsilon} +  \Theta\left(\frac{ 1}{\sqrt{d}}\right)}.
\end{equation}
Setting $2^{1/p}55^{-1/6} = 1$,
we find the solution
$p_0 = (6 \log 2)/\log 55 \approx 1.03782$. Observing that
$c_{p,d,\mu}$ does not depend on our choice of $\varepsilon$,
the result follows by letting $\varepsilon\to 0$.
\end{proof}

\begin{remark} For $p\le 1.03$, $2^{1/p}55^{-1/6} > 1.005$.
Thus, if $d$ is ``high", $c_{p,d,\mu}\ge 1.005^d$. How high
must $d$ be can be explicitly determined from the proof, by
keeping track of the constants in Lemma \ref{sphericalcaps}, instead
of writing $\Theta(1/\sqrt{d})$. Note also that in the specific case $p=1$,
the preceding theorem is more general
and  gives a better
bound (since $55^{-1/6} 2 > (2/\sqrt3)^{1/6}$) than \cite[Theorem 2.3]{A1}, even though $\chi_{B(0, R_1/2)}$ is a very poor choice when $p=1$
(using $\delta_0$ is much more efficient). We shall explore the
case $p=1$ in more detail elsewhere.
\end{remark}

\begin{remark}\label{td} The hypotheses contained in (\ref{hipotesis}) are selected
so that all finite, radial, radially decreasing measures with bounded
densities are included, and still
a concrete range for $p$ is obtained.
Numerically, $t_0:= \frac{6\log 2 - \log55}{3\log 3 - 3\log2}\approx 0.1246$. Provided that the singularity at $0$ is not
too strong, Theorem \ref{theoremdecp} also applies to measures with unbounded densities. In particular, it applies to all measures  defined
via (\ref{defrad}), with $f_t(r) =  r^{- t d} \chi_{(0,1]}(r)$ and
$t \in (0, 1 - t_0]$. This last condition comes from the fact that
for these measures,
$h_u (R) = u^{- (1 - t) d}$ when $R\le 1$.
\end{remark}

For infinite measures, however,
(\ref{hipotesis}) can be rather restrictive. Define  $\mu_{t,d}$ as in the preceding remark but
without truncation, i.e.,
using $f_t (r) =  r^{- t d}$. Then the theorem applies only
to $t= 1- t_0$. Observe, however, that values different
from $t_0$ and $\sqrt{2/3}$  could have been used, with
the same qualitative results.
Thus, a simple way to obtain a theorem covering an infinite subfamily of the measures
$\mu_{t,d}$
is to assume different rates of growth for the $\sup$ and the $\limsup$
in (\ref{hipotesis}).
The proof of the next result is essentially identical to that of Theorem \ref{theoremdecp},  so it will be omitted. We use  $u=\sqrt{2/3}$ to be able to apply the same splitting of the ball centered
at $R_1 e_1$, but other values are possible. Also, the upper bound given below for $t_1$ can be modified,
by suitably choosing a different value for $u$.
Recall that  $f: (0,\infty ) \to [0,\infty )$
is  nonincreasing and that  $\mu$ is
 defined by $f$ via (\ref{defrad}).

\begin{theorem}\label{theoremdecpgen} Fix
$d\in \mathbb{N}\setminus \{0\}$, choose $t_0 \in (0,1)$,
 $t_1\in (0,\log(64/55)/\log(9/4))$,
and set $u = \sqrt{2/3}$.
  Then there exists a $p_0 = p_0 (t_0, t_1) > 1$ with the following
  property: For all
  $p\in [1,p_0)$ we can find a $b(p, t_0, t_1) > 1$, such that for every
 measure $\mu$ satisfying
  $
   \sup_{R > 0} h_u (R) \ge
   u^{-t_0 d}
  $ and
  $  \limsup_{R\to \infty} h_u(R)
  \le u^{-t_1 d},
$
 we have
$
C_{p,d,\mu} \ge c_{p,d,\mu} \ge  \Theta\left(1\right) b(p, t_0, t_1)^{d}.
$
\end{theorem}

\begin{remark} If $t_0 < t_1$,
then the preceding result
 covers  all the measures $\mu_{t,d}$
defined by $f_t (r) =  r^{- t d}$ such that
$t_0\le 1 - t \le t_1$.
\end{remark}

Returning to Theorem \ref{theoremdecp}, it admits a simpler statement when $f$ is bounded
and $f (x) x^{d-1}\in L^1(0,\infty)$, so $\mu$ is finite.
By Proposition \ref{dimmeas}, the conditions
$\sup_{R > 0} h_u (R) \ge u^{-t_0 d}$ and
$\limsup_{R \to\infty} h_u (R) \le  u^{-t_1 d}$
are then automatically satisfied for all $t_0, t_1, u \in (0,1)$.

\begin{corollary}\label{corodecp} Fix
$d\in \mathbb{N}\setminus \{0\}$.
Suppose $f$ is bounded and $f (x) x^{d-1}\in L^1(0,\infty)$. If $\mu$ is the finite  measure defined via (\ref{defrad}), then
  for every
  \begin{equation}\label{largpbdd}
p \in \left[1, \frac{6\log 2}{\log 55}\right)
\mbox{\ \ \ \ \ we have \ \ \ \ \ }
C_{p,d,\mu} \ge c_{p,d,\mu} \ge  \frac{1}{4 + \Theta\left(\frac{1}{\sqrt{d}}\right)} \left(\frac{2^{1/p}}{55^{1/6}}\right)^{d}.
\end{equation}
\end{corollary}

\begin{example}\label{dec} When dealing with  concrete families of measures
it is possible  to obtain  tighter bounds.
We revisit the example from \cite[Remark 2.7]{A1}, adapting
the arguments given there to $p > 1$. Let $\nu_d (A) := \lambda^d (A\cap B^d)$ be Lebesgue
measure restricted to the unit ball.  We apply Lemma \ref{pbound} with $R=1$ and $v=1/2$, so $H = 2^{-1} \sqrt{5}$, and $B^d \cap B(e_1, H)$
is the union of two solid spherical caps, the largest of which
is $B^d \cap \{x_1 \ge 3/8\}$ (since the smaller sphere has larger curvature). Solving $x_1^2 + x_2 ^2 = 1$ and
$(x_1 - 1)^2 + x_2 ^2 = H^2$ we obtain
\begin{equation*}
 \nu_{d} (B( e_1, 2^{-1} \sqrt{5})) \le 2 \lambda^d (B^d \cap \{x_1 \ge
3/8\}) = 2 \lambda^{d-1} (B^{d-1})\int_{3/8}^1\left(\sqrt{1-x_1^2}\right)^{d-1}dx_1
\end{equation*}
\begin{equation*}
 \le \frac{16 \lambda^{d-1} (B^{d-1})}{3} \int_{\arcsin (3/8)}^{\pi/2}\cos^d t \sin t dt =  \frac{ 16}{3( d + 1)}\left( \frac{ \sqrt{55}}{8}\right)^{d+1} \lambda^{d-1} (B^{d-1}).
 \end{equation*}
By Lemma \ref{pbound},
\begin{equation}\label{stein}
 c_{p,d, \nu_d} \ge \left( \frac
{1}{2}\right)^{d/q} \frac{\lambda^{d} (B^{d})}{2 \lambda^{d-1} (B^{d-1})} \frac{3( d + 1)}{ 16} \left( \frac
{8}{\sqrt{55}}\right)^{d + 1} = \Theta (\sqrt{d}) \left( \frac
{ 2^{2 + 1/p}}{\sqrt{55}}\right)^{d}.
\end{equation}
Setting $2^{2 + 1/p} = \sqrt{55}$ and solving for $p$ we obtain that
$c_{p,d}$ grows exponentially fast with $d$
whenever
$$
p < \left(\frac{\log 55}{2 \log 2} - 2\right)^{-1}\approx 1.1227.
$$
\end{example}

\begin{remark} It is possible to present  more involved arguments in
Theorem \ref{theoremdecp} and in Example \ref{dec}, by
 trying to optimize in $v \in (0, 1]$ instead of simply using $v=1/2$. But
 this does not seem to
significantly improve the value of $p$. Specifically , using
$B(e_1, H)$ with $H = \sqrt{1 + v^2}$ in Example \ref{dec}, the same steps
followed above lead us to maximize
$
g(v, q) :=
\frac{2 v^{1/q}}{(3 + 2 v^2-v^4)^{1/2}}.
$
The particular choice $v = 1/2$ yielded the critical value
$p_0 \approx 1.1227$ and its conjugate exponent $q_0 \approx 9.1474$.
Now it is elementary to check  that for all $q \le 9$
and all $v\in [0,1)$,
$
g(v,q) < 1.
$
Thus,  with the methods of the present paper we cannot get
exponential increase in Example \ref{dec} for any $p\ge  9/8 = 1.125$,
which is very close to  $p_0 \approx 1.1227$. Even the general bound
$p_0 \approx 1.0378$
from Theorem  \ref{theoremdecp}  is not far from $1.125$.

 We mention
that although for small values of $q$, say, $q\approx 10$ it is better to
consider as our $L^p$ function $\chi_{B(0,v)}$ with $v\approx 1/2$,
in order to maximize $g(v,q)$
as $q\to\infty$, we must let $v\to0$. Of course, at the endpoint value $p=1$,
the Dirac delta measure $\delta_0$ is a better choice than all the functions
$\chi_{B(0,v)}$, $v\in (0,1)$.
\end{remark}

In general, good {\em upper} bounds for $c_{p,d}$ and $C_{p,d}$ are easier
to establish when $\mu$ is {\em doubling}, that is, when there exists
an absolute constant $C$ such that for all $x\in \mathbb{R}^d$
and all $R > 0$, $\mu (B(x, 2 R)) \le C \mu (B(x, R))$. The doubling
condition captures the property of Lebesgue measure that yields
weak type bounds via covering lemmas of Vitali type. It might
be
expected that arbitrary doubling measures would behave like Lebesgue measure, but in our
context this is not the case:
There is a collection of doubling measures for which exponential increase
holds even when $p$ is high. For all $t\in (0,1)$, let $\mu_{t,d}$ be defined on $\mathbb{R}^d$ by $d\mu_{t,d}:= \|x\|_2^{-td} dx$, and consider
the families $\mathcal{M}_t:= \{\mu_{t,d}: d\in \mathbb{N}\setminus \{0\}\}$.
It is well known that the measures $\mu_{t,d}$ are indeed doubling, cf. for instance \cite[2.7, p. 12]{St3}.  For simplicity, instead of
$C_{p,d,\mu_{t,d}}$ and  $c_{p,d,\mu_{t,d}}$ we shall
write $C_{p,d, t}$ and  $c_{p,d,t}$ to denote the best strong type and weak type $(p,p)$ constants for the measures
in $\mathcal{M}_t$.

\begin{theorem}\label{theoremdoubling} Fix
$p_0\in [1,\infty)$. Then there exist constants $t_0 = t_0(p_0)\in (0,1)$ and
 $b_0 = b_0(p_0) > 1$ such that for all $p\in [1,p_0]$ and all
$t\in [t_0, 1)$, we have
$
C_{p,d, t} \ge c_{p,d,t}\ge b_0^{(1 -t)d}/6.$
\end{theorem}

\begin{proof}  Assume that
$d \ge 2$. We use $\chi_{B(0, 1/2)}$ as our $L^p$ function,
taking  $R=1$ and  $H = 2^{-1}\sqrt{5}$ in Lemma \ref{pbound}.
Given $p_0$, we select a fixed  $c\in (1,2^{1/p_0})$, so  $2^{1/p_0}/c > 1$, and
split $B(e_1, H)$
 into the three sets
$B(0, c/2)\cap B(e_1, H)$,
$\left(B(0, 1)\setminus  B(0, c/2)\right)\cap B(e_1, H)$,
and $B(0, 1)^c\cap B(e_1, H)$. We shall select $t_0 = t_0 (c) < 1$ so
close to $1$ that for all $t \in [t_0, 1)$ the dominant term  will be
\begin{equation}\label{inside1doub}
\mu_{t,d} (B(0, c/2)\cap B(e_1, H))
\le
\mu_{t,d} (B(0, c/2))
=
\frac{\sigma^{d-1} (\mathbb{S}^{d -1}) \left(\frac{c}{2}\right)^{(1- t)d}}{(1- t)d},
\end{equation}
where the last equality is obtained by integrating in polar coordinates.
Solving $x_1^2 + x_2^2 = c^2/4$ and
$(x_1 - 1)^2 + x_2^2 = 5/4$, we find
that the middle section
$\left(B(0, 1)\setminus  B(0, c/2)\right)\cap B(e_1, H)$ is contained in the cone subtended by the cap
$C((c^2 - 1)/8,  c e_1/2)$, or, if we radially project this cap
into the unit sphere, by $C((c^2 - 1)/(4 c),  e_1)$. The radius of the smallest ball containing this cap is $x_2(c) = (4 c)^{-1} \sqrt{18 c^2 - c^4 - 1}$. Since $x_2(1) = 1$ and
$x_2$  is
strictly decreasing  on $[1,2]$
 (as can be seen from its geometric meaning or
by differentiating)  $x_2(c) < 1$
on $(1,2]$.
By
Lemma \ref{sphericalcaps},  $\sigma^{d-1}_N (C((c^2 - 1)/(4 c),  e_1)) = x_2 (c)^d /\Theta (\sqrt{d})$ and
\begin{equation}\label{middoub}
\mu_{t,d} \left(\left(B (0, 1)\setminus B(0, c/2)\right) \cap B( e_1, H)\right)
\le
\frac{\left(x_2 (c)^{1/(1-t)}\right)^{(1 -t) d}}{\Theta (\sqrt{d})} \frac{\sigma^{d-1} (\mathbb{S}^{d-1})}{(1 - t) d}.
\end{equation}
Solving
$x_1^1 + x_2^2 = 1$ and
$(x_1 - 1)^2 + x_2^2 = 5/4$ shows that $B(0, 1)^c\cap B(e_1, H)$ is contained
in the cone subtended by the cap $C(3/8, e_1)$.
The distance from
the boundary of this cap to the $x_1$ axis is $x_2 = \sqrt{55}/8$.
Thus,  Lemma \ref{sphericalcaps}, together with integration in polar
coordinates from $1$ to $1 + 2^{-1}\sqrt{5}$ in the radial variable
and over the said cap in the angular variable,
 yield
\begin{equation}\label{outdoub}
\mu_{t,d} (B(0, 1)^c\cap B(e_1,  H)) \le
\left(\left( \frac{ \sqrt{55}}{8}\right)^{1/(1-t)} \left(1 + 2^{-1}\sqrt{5}\right)\right)^{(1 -t) d} \frac{\sigma^{d-1} (\mathbb{S}^{d-1})}{(1 - t) d \Theta (\sqrt{d})}.
\end{equation}
Select $d_0=d_0(c)$ such that  if $d\ge d_0$, the expressions $1/\Theta (\sqrt{d})$ in (\ref{middoub}) and (\ref{outdoub})
coming from Lemma \ref{sphericalcaps} are  bounded  above by $1$.
As $t\to1$, both $x_2 (c)^{1/(1-t)}\to 0$ and $\left(\sqrt{55}/8\right)^{1/(1-t)}\to
0$, so by choosing $t_0$ close enough to $1$,
we can make the term in (\ref{inside1doub}) larger than those in
(\ref{middoub}) and (\ref{outdoub})
for every $d \ge d_0$ and all $t\in [t_0, 1)$. Hence,
\begin{equation}\label{totalH}
\mu_{t,d} (B(e_1, H))
\le
 3 \mu_{t,d} (B(0, c/2)) =
\frac{3 \sigma^{d-1} (\mathbb{S}^{d -1}) c^{(1- t)d}}{(1- t)d 2^{(1- t)d}}.
 \end{equation}
By Lemma \ref{pbound}
we obtain
$
c_{p,d, t}
\ge
\frac{1}{6} \left(\frac{1}{2}\right)^{(1-t) d/q}
\left(\frac{2}{c}\right)^{(1- t)d}
=
\frac{1}{6}
\left(\frac{2^{1/p}}{c}\right)^{(1- t)d}
\ge
\frac{1}{6}
\left(\frac{2^{1/p_0}}{c}\right)^{(1- t)d}
$
for every $p\in [1,p_0]$. Finally, by the choice of $c$, the inequality $b_0:=\min\{6^{1/d_0},c^{-1} 2^{1/p_0}\}>1$ holds, so
 we have exponential increase of the best constants in $d$ (the term $6^{1/d_0}$ has been included
 to account for small values of $d$).
\end{proof}

\begin{example}\label{radial} Define $\mu_v$ on $\mathbb{R}^d$ as the sum of area measure
on $\mathbb{S}$ plus Lebesgue
on $B^d(0,v)$. If $v=1$, then the arguments used in this paper apply and we do get exponential growth of $c_{p,d}$ for sufficiently
 small values of $p > 1$.  However, suppose we let $v \to 0$;
taking $\chi_{B(0,r)}$ as our $L^p$ function, we see that having
$v < r < 1$ offers no advantage over $0 < r \le v$, so
$r\to0$ as $v \to 0$.
This forces us to let $q = q(r) \to \infty$ in
 Lemma \ref{pbound}, and we do not obtain a uniform value of $p$ for this
 family.
While the measures $\mu_v$ are not absolutely continuous,
by taking $f_v = \chi_{[0,v]} + \chi_{[1 - 1/d,1]}$ we observe
the same phenomenon for densities. Thus,  additional hypotheses
are needed in order to go
beyond  radially decreasing densities.
\end{example}

\end{document}